\documentclass{article}
\usepackage{latexsym, amsmath, amssymb}
\usepackage{graphicx}
\usepackage{epstopdf}
\usepackage{amsthm, amscd, amsfonts}
\usepackage{array}
\usepackage{multirow}
\usepackage{float}
\usepackage{blkarray}
\usepackage{enumerate}
\usepackage{listings}
\usepackage[a4paper, total={6in, 10in}]{geometry}
\usepackage{tikz}

\tikzset{
	mynode/.style={fill,circle,inner sep=2pt,outer sep=0pt}
}\usepackage{tikz}

\tikzset{
	mynode/.style={fill,circle,inner sep=2pt,outer sep=0pt}
}

\newtheorem{thm}{Theorem}

\theoremstyle{theorem}

\makeatletter
\newcommand*{\rom}[1]{\expandafter\@slowromancap\romannumeral #1@}
\makeatother

\def\bege{\begin{equation}} \def\ende{\end{equation}}

\def\begr{\begin{eqnarray}}
	\def\endr{\end{eqnarray}} 
 
\def\bege{\begin{equation}} \def\ende{\end{equation}}
\def\begr{\begin{eqnarray}} \def\endr{\end{eqnarray}}
\def\bnum{\begin{enumerate}} \def\enum{\end{enumerate}}
\begin{document}
 \begin{center}
\Large{\textbf{On $k$-distance degree based topological indices of benzenoid systems}}
\end{center}
\begin{center}
Sohan Lal $^{1}$, Karnika Sharma$^{2}$, Vijay Kumar Bhat$^{3,}$$^{\ast}$
\end{center}

\begin{center}
School of Mathematics,\end{center} \begin{center}Shri Mata Vaishno
Devi University,\end{center}\begin{center}Katra-$182320$, Jammu and
Kashmir, India.
\end{center}
\begin{center}
1. sohan1993sharma@gmail.com 2. karnikasharma069@gmail.com 3. vijaykumarbhat2000@yahoo.com \underline{}
\end{center}
\vspace*{0.6cm}
\textbf{Abstract:}
Topological indices are graph invariants numeric quantities, which are utilized by researchers to analyze a variety of physiochemical aspects of molecules. The goal of developing topological indices is to give each chemical structure a numerical value while maintaining the highest level of differentiation. Using these indices, the classification of various structures, and their physiochemical and biological properties can be predicted. In this paper, the leap and leap hyper Zagreb indices, as well as their polynomials for a zigzag benzenoid system $Z_{p}$ and a rhombic benzenoid system $R_{p}$ are determined. In addition, new $k$-distance degree-based topological indices such as leap-Somber index, hyper leap forgotten index, leap $Y$ index, and leap $Y$ coindex are also computed for the molecular graphs of $Z_p$ and $R_p$. Furthermore, their numerical computation and discussion are performed to determine the significance of their physiochemical properties.\\


\section{Introduction}

\noindent Michael Faraday \cite{a} was the first person to isolate and identify benzene. Benzene is a very poisonous parent aromatic chemical with a colourless appearance and a distinct odour. Benzene is one of the top 20 compounds in terms of production volume, with applications including synthetic fibres, dyes, plastic, resins, polystyrene, medicines, and insecticides. The most common use of benzene is in the production of phenol. The toxicity of benzene was discovered in the blood-forming organs soon after it was first used in industry \cite{22}. Since its discovery, benzene's structure has attracted great interest. Despite several research efforts, there are still lots of questions concerning how to use benzene effectively. The discovery of benzene by Michael Faraday \cite{a} in 1825, growing the ability to interact with the matter of benzene.\\

\noindent A chemical compound that contains at least one benzene ring is called benzenoid and is typically used as an intermediary in the synthesis of other chemicals. A benzenoid is a colorless aromatic chemical compound that comprises of at least one benzene ring. Benzenoid hydrocarbons are important in the food and chemical industries, as well as in our environment. Benzenoid molecular graphs are systems in which hydrogen atoms have been removed. It is a connected geometric shape made up of congruent regular hexagons arranged in a plane so that two hexagons are either disjoint or share a common edge see Figure $1$. A hexagonal system's vertex can only belong to three hexagons. A vertex that is shared by three hexagons is an internal vertex. As benzenoid compounds become more important in practically all production processes across a wide range of industries, so the problem of computing benzenoid topological indices has attracted the interest of experts.\\

\begin{figure}[h!]
  \centering
  \includegraphics[width=3.5in]{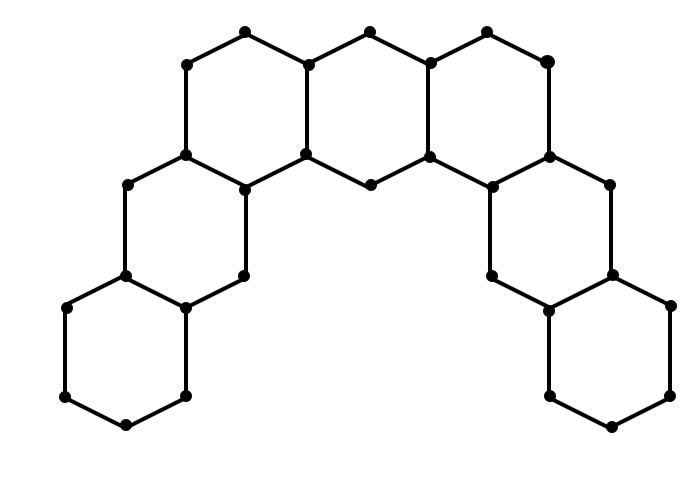}
  \caption{Benzenoid System}\label{Figure1}
\end{figure}

\noindent In this work, we exclusively look at molecular graphs. Molecular graphs are a visual representation of a chemical molecule, with vertices representing atoms and edges representing bonds between the atoms. Graph invariants connected with the graphs related to the chemical molecule can be used to study the physicochemical properties of a molecule. Topological index is one of these graph invariants. A topological index is a real number that is associated with a graph and is determined by a rule that is invariant for isomorphic graphs. It describes the molecular structure's topology and has advanced chemical applications in $QSPR\setminus QSAR$ studies \cite{kk,kkk}. In $1947$, chemist Harold Wiener invented the Wiener index \cite{27}, which began the history of topological indices. In the literature, many topological indices have been defined. Few of them are the Zagreb indices $Z(G)$ \cite{29}, the $F$-index $F(G)$ \cite{17}, Hyper $F$-index $HF(G)$ \cite{17}, and the Somber index $SO(G)$ \cite{ss}. Some recently introduced topological indices are $Y$-index $Y(G)$ \cite{16}, $Y$-coindex $\overline{Y}(G)$ \cite{21}, second hyper-Zagreb index $HM_{2}(G)$ and its coindex $\overline{HM_{2}}(G)$ \cite{16} and are defined as follows:\\

\begin{eqnarray*}
F(G)&=& \sum_{u \in V(G)}deg_{G}^{3}(u)\\
    &=& \sum_{uv\in E(G)}[deg_{G}^{2}(u)+deg_{G}^{2}(v)]\\
HF(G)&=& \sum_{uv\in E(G)}[deg_{G}^{2}(u)+deg_{G}^{2}(v)]^{2}\\
Y(G)&=& \sum_{u \in V(G)}deg_{G}^{4}(u)\\
    &=& \sum_{uv\in E(G)}[deg_{G}^{3}(u)+deg_{G}^{3}(v)]\\
  \overline{Y}(G) &=& (n-1)F(G)-Y(G)\\
  HM_{2}(G) &=& \sum_{uv\in E(G)}deg_{G}^{2}(u)deg_{G}^{2}(v) \\
   \overline{HM_{2}}(G) &=& \sum_{uv\notin E(G)}deg_{G}^{2}(u)deg_{G}^{2}(v)\\
   SO(G)&=& \sum_{uv\in E(G)}\sqrt{deg_{G}(u)^2+deg_{G}(v)^2}
\end{eqnarray*}
\noindent Wiener proposed the first distance based topological index in 1947 \cite{27} while researching on the paraffin breaking point, whereas Platt proposed the first degree based topological index in 1947 for predicting physical features of alcanes \cite{28}. The Zagreb index serves as the foundation for the first-order connectivity index. More often than any other topological indicator, the connectedness index and its variations are utilised in QSPR and QSAR \cite{A, B, C}. The Zagreb indices and their variations have recently been used to explore chirality, chemical complexity, heterosystems, and ZE-isomerism, while the overall Zagreb indices showed promise for application in the development of multilinear regression models. Numerous researchers also employ Zagreb indices in their QSPR and QSAR investigations \cite{A, B, C}. Interested readers can refer to \cite{18} for more information about Platt index and Zagreb indices.\\ 

\noindent In \cite{22} Gao et al. calculated exact formulas for the Zagreb and hyper-Zagreb indices of Carbon Nanocones $CNC_{k}[n]$, and defined a new degree-based topological index called the second hyper-Zagreb index. Nazeran et al. \cite{44} gave the exact relations for first and second Zagreb index, hyper Zagreb index, multiplicative
Zagreb indices as well as first and second Zagreb polynomials for some benzenoid systems.\\

\noindent Fartula et al. \cite{17} established a forgotten topological index and its properties. The prediction power of the forgotten index was examined using a dataset of octane isomers, in accordance with advice from the International Academy of Mathematical Chemistry. The forgotten index does not identify heteroatoms and multiple bonds in its most basic form, that's why this dataset was chosen. Boiling point, melting point, heat capacities, entropy, density, heat of vaporisation, enthalpy of formation, motor octane number, molar refraction, acentric factor, total surface area, octanol-water partition coefficient, and molar volume are among the properties included in the octane dataset. Each of these attributes was connected with the forgotten index, and the outcomes were compared to those attained using the first Zagreb index. It was discovered that the forgotten-predictive index's power is pretty comparable to first zagreb index. For the entropy and acentric factor they produce correlation coefficients of more than 0.95. On the other hand, neither first zagreb nor forgotten index provide a satisfactory correlation for other physicochemical properties.\\
 
\noindent In \cite{ss}, Gutman listed 26 topological indices, two of which were Zagreb indices. In the same publication, Gutman introduced a novel method for computing the (molecular) graphs' vertex-degree-based topological index. Moreover, he found that the theory of vertex-degree-based topological indices had never employed the function $F(x,y) = \sqrt{x^2+y^2}$ and introduced a somber index as\\
\begin{eqnarray*}
	SO(G)&=& \sum_{uv\in E(G)}\sqrt{deg_{G}(u)^2+deg_{G}(v)^2}
\end{eqnarray*}
Motivated by the Gutman's work, K. C. Das et al., \cite{tt} obtained some mathematical properties like upper and lower bounds on the somber index and characterized extremal graphs. Also, they determine the relation between the somber index and the first and second Zagreb indices. The somber indices have strong predictive applicability in mathematical chemistry. A. Ulker et al. \cite{A} obtained bounds for graph energy in terms of its Sombor index and provided the relation between Sombor index and graph energy for some graph classes. After the introduction of somber index a lot of research can be done on properties of the somber index. For basic, general, and mathematical properties of somber index see \cite{uu,vv}. Chemical applicability of somber index was obtained in \cite{ww}.\\

\noindent One of the most significant Topological coindices in quantitative structure-property relationships is the Y-coindex (QSPR). Without conducting a chemical experiment, Y-coindex can provide medical information about some new medications. One of the best correlation indices for understanding the physicochemical characteristics of octane isomers is the Y-coindex. To know more about the Y-coindex see \cite{21,16}.\\

\noindent Naji et al., introduced $k$-distance degree based topological indices in \cite{xx} and presented a explicit formulae for a cartesian product of a graphs which were applicable to calculate $k$-distance degree index of some chemical and computer science based graphs like hypercube, hamming graphs, nanotubes and nanotori. Naji et al. \cite{1} proposed the leap Zagreb indices of a graph $G$, which are $k$-distance degree based topological indices that are based on the second degree of vertices. They also identify the chemical applications of first leap Zagreb index. Moreover, they found that this index has very good correlation with physical properties of chemical compounds like boiling point, entropy, DHVAP, HVAP and accentric factor. One can refer to \cite{6,15} for further information on Naji et al's work on the leap Zagreb indices.\\

\noindent The idea of $k$-distance has been applied in many other applicable settings. Multi-hop generalized core percolation on complex networks ia a key reference \cite{yy} as the multi-hop idea is same like the $k$-distance. The problem of computing distance-degree based topological indices has attracted the interest of researchers for many years. J. M. Zhu et al. \cite{31} established a lower bound on the third leap Zagreb index for trees. Z. Shao et al. \cite{5} derived leap Zagreb indices for trees and unicyclic networks. Chidambaram et. al \cite{kkkk} computed the leap Zagreb indices of bridge and chain graphs. For some recent work on topological indices, interested readers can refer to \cite{21,1,16,40}. Despite the work that has already been done on $k$-distance degree index, there is still much research to be done.\\

\noindent The above-mentioned study and applications of topological indices in the field chemistry and biology inspired us to develop some newly $k$-distance degree based topological indices such as leap somber index, hyper leap forgotten index, leap $Y$ index, and leap $Y$ coindex. Therefore, in this paper, we compute the leap and hyper leap Zagreb indices, as well as their polynomials for the molecular graph of zigzag benzenoid system $Z_{p}$ and rhombic benzenoid system $R_{p}$. In addition, the leap somber index, hyper leap forgotten index, leap $Y$ index, and leap $Y$ coindex are also calculated for these molecular graphs. Furthermore, their  numerical computations and verification were also performed.
\section{Preliminaries}
\noindent In this section, we focused on some basic and new concepts related to the $k$-distance degree based topological indices.\\
\noindent Let $G$ be a simple molecular graph with the vertex set $V(G)$, the edge set $E(G)$. For every vertex $v\in V(G)$, the degree of $v$ is defined as $deg(v)=|\{u\in V(G):uv\in E(G)\}|$. The distance between any two vertices $u$, $v$ of $G$ is denoted by $d(u,v)$ and is defined as the length of the shortest path joining $u$ and $v$. The open $k$-neighborhood of a vertex $v$ is represented by $N_{k}(v/G)$ and is defined as $N_{k}(v/G)=\{u\in V(G):d(u,v)=k, \forall k\in \mathbb{Z}^+\}$. The $k$-degree of a vertex $v$ in $G$ is denoted by $deg_{k}(v)$ and is defined as the number of $k$-neighbors of the vertex $v$ in $G$, i.e., $deg_{k}(v)=|N_{k}(v/G)|$ \cite{1}. In this study we are computing $k$-distance degree based topological indices for $k=2$.\\

\noindent A new version of first leap Zagreb index was defined by V.R. Kulli \cite{10}.
\begin{eqnarray}
  LM_{1}(G) &=& \sum_{uv\in E(G)}[deg_{2}(u)+deg_{2}(v)]
\end{eqnarray}
\noindent The second leap Zagreb index was introduced by Naji et al. \cite{2}.
\begin{eqnarray}
  LM_{2}(G) &=& \sum_{uv\in E(G)}deg_{2}(u)deg_{2}(v)
\end{eqnarray}
\noindent Further, V. R. Kulli \cite{10} proposed two more topological indices named as first and second hyper leap-Zagreb indices and are defined as
\begin{eqnarray}
  HLM_{1}(G) &=& \sum_{uv\in E(G)}[deg_{2}(u)+deg_{2}(v)]^{2}\\
  HLM_{2}(G) &=& \sum_{uv\in E(G)}[deg_{2}(u)deg_{2}(v)]^{2}
\end{eqnarray}
\noindent Furthermore, by considering leap zagreb indices, V. R. Kulli \cite{10} introduced the first and second leap Zagreb polynomials as well as the first and second leap hyper-Zagreb polynomials for a graph $G$, respectively.
\begin{eqnarray}
  LM_{1}(G:x) &=& \sum_{uv\in E(G)}x^{[deg_{2}(u)+deg_{2}(v)]}\\
  LM_{2}(G:x) &=& \sum_{uv\in E(G)}x^{[deg_{2}(u)deg_{2}(v)]}\\
  HLM_{1}(G:x) &=& \sum_{uv\in E(G)}x^{[deg_{2}(u)+deg_{2}(v)]^{2}}\\
  HLM_{2}(G:x) &=& \sum_{uv\in E(G)}x^{[deg_{2}(u)deg_{2}(v)]^{2}}
\end{eqnarray}
\noindent Now, we define new $k$-distance degree topological indices that are leap Somber index, hyper leap forgotten index, leap $Y$ index, and leap $Y$ coindex as:
\begin{eqnarray}
  LSO(G) &=& \sum_{uv\in E(G)}\sqrt{deg_{k}(u)+deg_{k}(v)}\\
HLF(G) &=& \sum_{uv\in E(G)}[deg_{k}^{2}(u)+deg_{k}^{2}(v)]^{2}\\
  LY(G) &=& \sum_{uv\in E(G)}[deg_{k}^{3}(u)+deg_{k}^{3}(v)]\\
 \overline{LY}(G) &=& (n-1)(LF(G))-(LY(G))
\end{eqnarray}
where,
\begin{eqnarray}
LF(G) &=& \sum_{uv\in E(G)}[deg_{k}^{2}(u)+deg_{k}^{2}(v)]
\end{eqnarray}

\noindent Throughout the work, we calculate distance degree based topological indices, their polynomials, and newly defined k-distance degree based topological indices for the molecular structures of zigzag benzenoid system $Z_{p}$ and rhombic benzenoid system $R_{p}$. Furthermore, their numerical computations and verifications were performed. 	
\section{Main Results}
\noindent In this section, we provide some results related to pre-defined distance degree based topological indices such as first and second leap Zagreb indices, first and second leap hyper Zagreb indices as well as their polynomials for the molecular structures of $Z_{p}$ and $R_{p}$. Also, we compute some results related to newly defined $k$-distance degree based topological indices for $Z_p$ and $R_p$. We consider $k=2$ and number of edges as a frequency of edges throughout the study.
\subsection{Zigzag Benzenoid System}
In this subsection, the detail structure and the edge partition technique for the molecular structure zigzag benzenoid system $Z_p$ and their related results are discussed.\\

\noindent Figure 2 shows the graph of a zigzag benzenoid system $Z_{p}$. The number of rows in $Z_{p}$ is $p$, and each row of the system has two hexagons with $8p+2$ vertices and $10p + 1$ edges.\\
\begin{figure}[h!]
	\centering
	\includegraphics[width=4in]{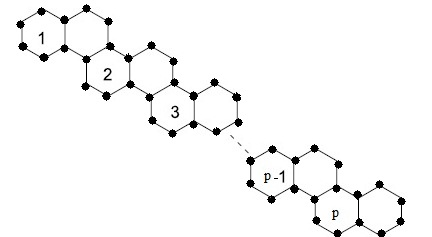}
	\caption{Zigzag benzenoid system $Z_{p}$}
\end{figure}

\noindent We divide the edge set by the distance degree of endpoints of the graph's edges. Then, we compute the number of edges (frequencies) in each set of the partition. All of the vertices have distance two from other vertices is considered to be the degree of endpoints vertices. The edges whose one endpoint vertex have two vertices of distance two and other endpoint vertex also have two vertices of distance two are only found on one upside and one downside of the chain. That means, we have only two such edges in the entire chain. Next, we analyse the number of edges for endpoints $deg_{2}(u)=2,deg_{2}(v)=3$, $deg_{2}(u)=3,deg_{2}(v)=4$ are four having two edges upside and rest are downside of the chain. Similarly, we found the number of edges for the endpoints $deg_{2}(u)=3,deg_{2}(v)=3$, $deg_{2}(u)=3,deg_{2}(v)=5$, $deg_{2}(u)=4,deg_{2}(v)=5$, $deg_{2}(u)=5,deg_{2}(v)=5$ as $2(p-1)$, $4(p-1)$, $p$, $3(p-1)$, respectively. Thus, the edge partition of $Z_p$ is summarized in Table 1.
\begin{table}[h!]
	\centering
\caption{Edge partition of $Z_{p}$} \label{table2}
\setlength{\tabcolsep}{30pt}
\renewcommand{\arraystretch}{1.0}
\begin{tabular}{l ccc}
\hline
$deg_{2}(u)$ & $deg_{2}(v)$ & frequency\\
\hline
2 & 2 & 2\\
2 & 3 & 4\\
3 & 3 & 2(p-1)\\
3 & 4 & 4\\
3 & 5 & 4(p-1)\\
4 & 5 & p\\
5 & 5 & 3(p-1)\\
\hline
\end{tabular}
\end{table}

\noindent First the result is obtained for pre defined first and second leap Zagreb indices, first and second leap hyper Zagreb indices for the molecular structure $Z_p$, $p\geq 2$.
\begin{thm}
If $G$ be a molecular graph of a zigzag benzenoid system $Z_{p}$, $p\geq2$. Then

  \item (i) $LM_{1}(G)=83p-18$
  \item (ii) $LM_{2}(G)=173p-73$
  \item(iii) $HLM_{1}(G)=709p-350$
  \item (iv) $HLM_{2}(G)=3337p-2257$

\end{thm}
\begin{proof}
Let $G$ be the molecular graph of zigzag benzenoid system $Z_p$, $p\geq2$. Then $|V(G)|=8p+2$ and $|E(G)|=10p+1$. Now, by using Table $1$ and equations $1-4$, we compute\\
\item \textbf{(i) First leap Zagreb Index}
	\begin{eqnarray*}
LM_1(G)&=&\sum\limits_{uv\in E(G)}[deg_2(u)+deg_2(v)]\\
         &=&2(2+2)+4(2+3)+4(3+4)+2(p-1)(3+3)+4(p-1)(3+5)+p(4+5)\\
         & &+3(p-1)(5+5)\\
         &=&2(4)+4(5)+4(7)+2(p-1)(6)+4(p-1)(8)+p(9)+3(p-1)(10)\\
         &=&8+20+28+(2p-2)(6)+(4p-4)(8)+p(9)+(3p-3)(10)\\
         &=&56+12p-12+32p-32+9p+30p-30\\
         &=&83p-18
\end{eqnarray*}
\item \textbf{(ii) Second leap Zagreb Index}
\begin{eqnarray*}
LM_2(G)&=&\sum\limits_{uv\in E(G)}[deg_2(u)deg_2(v)]\\
&=&2(4)+4(6)+4(12)+2(p-1)9+4(p-1)15+p20+3(p-1)25\\
&=&8+24+48+18(p-1)+60(p-1)+20p+75(p-1)\\
&=&80+18p-18+60p-60+20p+75p-75\\
&=&173p-73
\end{eqnarray*}
\item \textbf{(iii) First leap  hyper-Zagreb Index}
\begin{eqnarray*}
HLM_1(G) &=&\sum\limits_{uv\in E(G)}[deg_2(u)+deg_2(v)]^2\\
&=& 2(4)^2+4(5)^2+4(7)^2+2(p-1)(6)^2+4(p-1)(8)^2+p(9)^2+\\
& & 3(p-1)(10)^2\\
&=& 2(16)+2(25)+4(49)+2(p-1)(36)+4(p-1)(64)+81p+\\
& & 3(p-1)(100)\\
&=& 32+50+196+72(p-1)+256(p-1)+81p+300(p-1)\\
&=& 709p-350
\end{eqnarray*}
\item \textbf{(iv) Second leap hyper-Zagreb Index}
\begin{eqnarray*}
HLM_2(G) &=&\sum\limits_{uv\in E(G)}[deg_2(u)deg_2(v)]^2\\
&=& 2(4)^{2}+4(6)^{2}+4(12)^{2}+2(p-1)(9)^{2}+4(p-1)(15)^{2}+p(20)^{2}+\\
& & 3(p-1)(25)^{2}\\
&=& 2(16)+2(36)+4(144)+2(p-1)(81)+4(p-1)(225)+400p+\\
& & 3(p-1)(625)\\
&=& 32+72+576+162(p-1)+900(p-1)+400p+1875(p-1)\\
&=& 3337p-2257
\end{eqnarray*}
 \end{proof}

\noindent Next, on the basis of above mentioned result we obtain the first and second leap Zagreb polynomials as well as first and second leap hyper-Zagreb polynomials for $Z_p$, $p\geq 2$.
 \begin{thm}
Let $G$ be a molecular graph of zigzag benzenoid system $Z_p$, $p\geq2$. Then\\
\item(i) $LM_1(G:x)=2x^4+4x^5+2(p-1)x^6+4x^7+4(p-1)x^8+px^9+3(p-1)x^{10}$
\item(ii) $LM_2(G:x)=2x^4+4x^6+2(p-1)x^9+4x^{12}+4(p-1)x^{15}+px^{20}+3(p-1)x^{25}$
\item (iii) $HLM_1(G:x)=2x^{16}+4x^{25}+4x^{49}+4(p-1)x^{64}+px^{81}+3(p-1)x^{100}+2(p-1)x^{36}$
\item(iv) $HLM_2(G:x)=2x^{16}+4x^{36}+2(p-1)x^{81}+4x^{144}+4(p-1)x^{225}+px^{400}+3(p-1)x^{625}$

\end{thm}
\begin{proof}
From Table $1$ and equations $5-8$, we obtain the following polynomials for the molecular structure $Z_{p}$.\\
\item \textbf{(i) First leap Zagreb Polynomial}
	\begin{eqnarray*}
	LM_1(G:x)&=&\sum\limits_{uv\in E(G)}x^{[deg_2(u)+deg_2(v)]}\\
	           &=&2x^4+4x^5+2(p-1)x^6+4x^7+4(p-1)x^8+px^9+3(p-1)x^{10}
\end{eqnarray*}
\item \textbf{(ii) Second leap Zagreb Polynomial}
	\begin{eqnarray*}
LM_2(G:x)&=&\sum\limits_{uv\in E(G)}x^{[deg_2(u)deg_2(v)]}\\
	     &=&2x^4+4x^6+2(p-1)x^9+4x^{12}+4(p-1)x^{15}+px^{20}+\\
         & &3(p-1)x^{25}
\end{eqnarray*}
\item \textbf{(iii) First leap hyper-Zagreb Polynomial}
	\begin{eqnarray*}
HLM_1(G:x)&=&\sum\limits_{uv\in E(G)}x^{[deg_2(u)+deg_2(v)]^2}\\
	      &=&2x^{4^2}+4x^{5^2}+4x^{7^2}+4(p-1)x^{8^2}+px^{9^2}+3(p-1)x^{10^2}+\\
& & 2(p-1)x^{6^2}\\
	      &=&2x^{16}+4x^{25}+2(p-1)x^{36}+4x^{49}+4(p-1)x^{64}+px^{81}+\\
& & 3(p-1)x^{100}
\end{eqnarray*}
\item \textbf{(iv) Second leap hyper-Zagreb Polynomial}
		\begin{eqnarray*}
HLM_2(G:x)&=&\sum\limits_{uv\in E(G)}x^{[deg_2(u)deg_2(v)]^2}\\
	      &=&2x^{16}+4x^{36}+2(p-1)x^{81}+4x^{144}+4(p-1)x^{225}+px^{400}+\\
          & & 3(p-1)x^{625}
\end{eqnarray*}
\end{proof}
\noindent Now, the newly defined $k$-distance degree based topological indices such as the leap somber index, hyper leap forgotten index, leap $Y$ index, and leap $Y$ coindex are obtained for the molecular structure $Z_p$, $p\geq 2$. 
\begin{thm}
Let $G$ be a molecular graph of Zigzag benzenoid system $Z_p$, $p\geq2$. Then

\item(i) $LSO(G)=4(1+\sqrt 5+\sqrt{7})+(p-1)(8\sqrt{2}+3\sqrt 10+2\sqrt{6})+3p$
\item(ii) $HLF(G)=14453p-9468$
\item(iii) $LY(G)=1655p-930$
\item (iv)$\overline{LY}(G)=-1292p^2+2068p-776.$

\end{thm}
\begin{proof}
Utilising edge partition from Table $1$, we calculate the new $k$-distance degree based topological indices defined in equations $9-13$ as:

\item \textbf{(i) Leap Somber Index}
\begin{eqnarray*}
LSO(G) &=& \sum_{uv\in E(G)}\sqrt{deg_{2}(u)+deg_{2}(v)}\\ &=&2(4)^{\frac{1}{2}}+4(5)^{\frac{1}{2}}+4(7)^{\frac{1}{2}}+4(p-1)(8)^{\frac{1}{2}}+p(9)^{\frac{1}{2}}+3(p-1)(10)^{\frac{1}{2}}+\\
    & & 2(p-1)(6)^{\frac{1}{2}}\\
  	&=&4 +4\sqrt{5}+4\sqrt{7}+8(p-1)\sqrt{2}+3p+3(p-1)\sqrt{10}+2(p-1)\sqrt{6}\\
  	&=& 4(1+\sqrt 5+\sqrt{7})+(p-1)(8\sqrt{2}+3\sqrt 10+2\sqrt{6})+3p
\end{eqnarray*}
\item \textbf{(iii) Hyper leap forgotten Index}
\begin{eqnarray*}
HLF(G)&=&\sum\limits_{uv\in E(G)}[deg^2_2(u)_+deg^2_2(v)]^2\\
	&=&2(4+4)^2+4(4+9)^2+4(9+16)^2+4(p-1)(9+25)^2+p(16+25)^2\\
    & & +3(p-1)(25+25)^2+2(p-1)(9+9)^2\\
	&=&2(64)+4(169)+4(625)+4(p-1)(1156)+1681p+3(p-1)(2500)+\\
    & & 2(p-1)(324)\\
	&=&128+676+2500+4624(p-1)+1681p+7500(p-1)+648(p-1)\\
	&=&3304+(p-1)(4624+7500+648)+1681p\\
	&=&14453p-9468
\end{eqnarray*}
\item \textbf{(iv) Leap Y Index}
\begin{eqnarray*}
LY(G)&=&\sum\limits_{uv\in E(G)}[deg^3_2(u)+deg^3_2(v)]\\
	&=&2(8+8)+4(8+27)+4(27+64)+4(p-1)(27+125)+p(64+125)\\
& & +3(p-1)(125+125)+2(p-1)(27+27)\\
&=&32+140+364+608(p-1)+189p+750(p-1)+108(p-1)\\
&=&536+1466(p-1)+189p\\
&=&1655p-930
\end{eqnarray*}
\item \textbf{(v) Leap Y co-index}
\begin{eqnarray*}
\overline{LY}(G)&=&(p-1)(LF(G))-(LY(G))\\
	&=&pLF(G)-pLY(G)-LF(G)+LY(G)\\
	&=&p(363p-154)-p(1655p-930)-(363p-154)+(1655p-930)\\
	&=&363p^2-154p-1655p^2+930p-363p+154+1655p-930\\
	&=&-1292p^2+2068p-776.
	\end{eqnarray*}
where, $LF(G) = 363p-154$.
\end{proof}
\subsection{Rhombic Benzenoid System.}
\noindent In this subsection, we discuss the whole structure and the edge partition technique of the molecular structure of rhombic benzenoid system $R_p$. Furthermore, the results related to the distance degree, $k$-distance degree based topological indices and distance degree based polynomials are obtained.\\

\noindent Consider a benzenoid system in which hexagons are arranged to form a rhombic shape $R_{p}$, where $p$ denotes the number of hexagons along each rhombic boundary is shown in Figure 3. There are $2p(p + 2)$ vertices and $3p^2 + 4p – 1$ edges in this benzenoid system.\\
\begin{figure}[h!]
  \centering
  \includegraphics[width=3in]{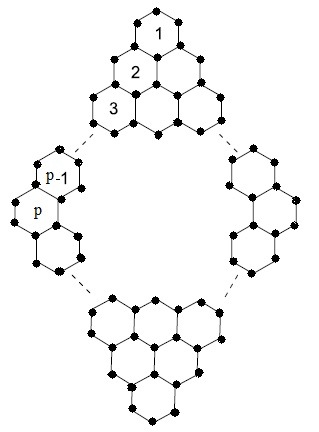}
  \caption{Rhombic Benzenoid System, $R_{p}$}\label{Figure3}
\end{figure}

\noindent We partitioned the edge set into three groups based on the distance degree of each edge's end vertices. The edges having $deg_{2}(u)=2,deg_{2}(v)=3$ appear on the extreme upside and downside corner of the rhombus, which are exactly $4$. On the outer boundary of the rhombus $deg_{2}(u)=3,deg_{2}(v)=3$ lies on the outer middle corner of the left and right side of the rhombus and $deg_{2}(u)=3,deg_{2}(v)=4$ lies on the upper and lower side of the outer middle corner as well as on the both side of the extreme upside and downside corner of the rhombus, which are counted as exactly $2$ and $8$, respectively. Whereas, the edges having $deg_{2}(u)=4,deg_{2}(v)=4$, $deg_{2}(u)=4,deg_{2}(v)=6$ and $deg_{2}(u)=6,deg_{2}(v)=6$ are appearing on the inner filling of the rhombus except $deg_{2}(u)=4,deg_{2}(v)=4$ which lies on the outer boundary of the rhombus and exist if $p>2$ and their frequencies are like $8(p-2)$, $4(p-1)$ and $(p-1)^{2}+2(p-1)(p-2)$. Table $2$ summarize the edge partition of rhombic benzenoid system $R_p$.
\begin{table}[h!]
	\centering
\caption{Edge partition of $R_{p}$} \label{table2}
\setlength{\tabcolsep}{30pt}
\renewcommand{\arraystretch}{1.0}
\begin{tabular}{l ccc}
\hline
$deg_{2}(u)$ & $deg_{2}(v)$ & frequency\\
\hline
2 & 3 & 4\\
3 & 3 & 2\\
3 & 4 & 8\\
4 & 4 & 8(p-2)\\
4 & 6 & 4(p-1)\\
6 & 6 & $(p-1)^{2}+2(p-1)(p-2)$\\
\hline
\end{tabular}
\end{table}

\noindent Next, the distance degree based topological indices such as first and second leap Zagreb index asnwell as the first and second leap hyper Zagreb index for the molecular structure $R_p$, $p\geq 2$ are obtained.
\begin{thm}
If $G$ be a molecular graph of rhombic benzenoid system $R_p$, $p\geq2$. Then

 \item (i) $LM_1(G)=36p^2+8p-20$
\item (ii) $LM_2(G)=108p^2-64p-34$
\item (iii) $HLM_1(G)=432p^2-240p-140$
\item (iv) $HLM_2(G)=3888p^2-6016p+1538$

\end{thm}
\begin{proof}
Let $G$ be a molecular graph of rhombic benzenoid system $R_p$, $p\geq2$. Then $|V(G)|=2p(p+2)$ and $|E(G)|=3p^2+4p-1$. Using edge partition from Table $2$, we compute different indices for $R_p$ defined in equations $1-4$.
\item \textbf{(i) First leap Zagreb Index}
	\begin{eqnarray*}
LM_1(G)&=&\sum\limits_{uv\in E(G)}[deg_2(u)+deg_2(v)]\\
	&=&4(5)+2(6)+8(7)+8(p-2)(8)+4(p-1)(10)+((p-1)^2+\\
& &2(p-1)(p-2))(12)\\
	&=&20+12+56+64(p-2)+40(p-1)+12(3p^{2}-8p+5)\\
	&=&88-128-40+60+p(64+40-96)+36p^{2}\\
	&=&36p^{2}+8p-20
\end{eqnarray*}
\item \textbf{(ii) Second leap Zagreb Index}
	\begin{eqnarray*}
LM_2(G)&=&\sum\limits_{uv\in E(G)}[deg_2(u)deg_2(v)]\\
	&=&4(6)+2(9)+8(12)+8(p-2)(16)+4(p-1)(24)+((p-1)^2+\\
& &2(p-1)(p-2))36\\
	&=&24+18+96+128(p-2)+96(p-1)+36(3p^{2}-8p+5)\\
    &=&138-256-96+180+p(128+96-288)+108p^{2}\\
    &=&108p^{2}-64p-34
\end{eqnarray*}
\item \textbf{(iii) First leap hyper-Zagreb Index}
\begin{eqnarray*}
HLM_1(G)&=&\sum\limits_{uv\in E(G)}[deg_2(u)+deg_2(v)]^2\\
	   &=&4(5)^{2}+2(6)^{2}+8(7)^{2}+8(p-2)(8)^{2}+4(p-1)(10)^{2}+((p-1)^2\\
       & &+2(p-1)(p-2))(12)^{2}\\
    &=&4(25)+2(36)+8(49)+8(p-2)(64)+4(p-1)(100)+((p-1)^2\\
    & &+2(p-1)(p-2))(144)\\
	&=&100+72+392+512(p-2)+400(p-1)+144(3p^{2}-8p+5)\\
	&=&564-1024-400+720+p(512+400-1152)+432p^{2}\\
	&=&432p^{2}-240p-140
\end{eqnarray*}
\item \textbf{(iv) Second leap hyper-Zagreb Index}
\begin{eqnarray*}
HLM_2(G) &=&\sum\limits_{uv\in E(G)}[deg_2(u)deg_2(v)]^2\\
&=&4(6)^{2}+2(9)^{2}+8(12)^{2}+8(p-2)(16)^{2}+4(p-1)(24)^{2}+((p-1)^2\\
& &+2(p-1)(p-2))36^{2}\\
	&=&144+162+1152+2048(p-2)+2304(p-1)+1296(3p^{2}-8p+5)\\
    &=&1458-4096-2304+6480+p(2048+2304-10368)+3888p^{2}\\
    &=&3888p^{2}-6016p+1538
\end{eqnarray*}	
\end{proof}
\noindent Further, the first and second leap Zagreb as well as the first and second leap hyper Zagreb polynomials for $R_p$, $p\geq 2$ are obtained.
\begin{thm}
Let $G$ be a molecular graph of rhombic benzenoid system $R_p$, $p\geq2$. Then

 \item(i) $LM_1(G:x)=4x^5+2x^6+8x^7+8(p-2)x^8+4(p-1)x^{10}+(3p^2-8p+5)x^{12}$
\item(ii) $LM_2(G:x)=4x^6+2x^9+8x^{12}+8(p-2)x^{16}+4(p-1)x^{24}+(3p^2-8p+5)x^{36}$
\item(iii) $HLM_1(G:x)=4x^{25}+2x^{36}+8x^{49}+8(p-2)x^{64}+4(p-1)x^{100}+(3p^2-8p+5)x^{144}$
\item(iv) $HLM_1(G:x)=4x^{36}+2x^{81}+8x^{144}+8(p-2)x^{256}+4(p-1)x^{576}+(3p^2-8p+5)x^{1296}$
\end{thm}
\begin{proof}
Now, we continuing to obtain the polynomials for $R_{p}$ mentioned in equations $5-8$ and are obtained using Table $2$.
\item \textbf{(i) First leap Zagreb Polynomial}
		\begin{eqnarray*}
		LM_1(G:x)&=&\sum\limits_{uv\in E(G)}x^{[deg_2(u)+deg_2(v)]}\\
		&=&4x^5+2x^6+8x^7+8(p-2)x^8+4(p-1)x^{10}+(3p^2-8p+5)x^{12}
	\end{eqnarray*}
\item \textbf{(ii) Second leap Zagreb Polynomial}	
			\begin{eqnarray*}
			LM_2(G:x)&=&\sum\limits_{uv\in E(G)}x^{[deg_2(u)deg_2(v)]}\\
			&=&4x^6+2x^9+8x^{12}+8(p-2)x^{16}+4(p-1)x^{24}+(3p^2-8p+5)x^{36}
		\end{eqnarray*}
\item \textbf{(iii) First leap hyper-Zagreb Polynomial}	
		\begin{eqnarray*}
			HL_1M(G,x)&=& \sum\limits_{uv\in E(G)}x^{[deg_2(u)deg_2(v)]^2}\\
			 &=& 4x^{5^2}+2x^{6^2}+8x^{7^2}+8(p-2)x^{8^2}+4(p-1)x^{10^2}+(3p^2-8p+5)x^{12^2}\\
		     &=&4x^{25}+2x^{36}+8x^{49}+8(p-2)x^{64}+4(p-1)x^{100}+(3p^2-8p+5)x^{144}
		\end{eqnarray*}
\item \textbf{(iv) Second leap hyper-Zagreb Polynomial}
	\begin{eqnarray*}
	HL_2M(G,x)&=&\sum\limits_{uv\in E(G)}x^{[deg_2(u)deg_2(v)]^2}\\
	&=&4x^{6^2}+2x^{9^2}+8x^{12^2}+8(p-2)x^{16^2}+4(p-1)x^{24^2}+(3p^2-8p+5)x^{36^2}\\
	&=&4x^{36}+2x^{81}+8x^{144}+8(p-2)x^{256}+4(p-1)x^{576}+(3p^2-8p+5)x^{1296}
\end{eqnarray*}
\end{proof}
\noindent Next, the newly defined $k$-distance degree based topological indices for the molecular structure rhombic benzenoid system $R_p$, $p\geq 2$ are obtained.
\begin{thm}
Let $G$ be a molecular graph of rhombic benzenoid system $R_p$, $p\geq2$. Then

\item (i) $LSO(G)=6p^2\sqrt{3}+p(16\sqrt{2}+4\sqrt{10}-16\sqrt{3})+4(\sqrt{5}-\sqrt{10})+2(\sqrt{6}+4\sqrt{7}+5\sqrt{3}-16\sqrt{2})$
\item (ii)$HLF(G)=15552p^2-22464p+5044$
\item (iii) $LY(G)=1296p^2-1312p-32$
\item (iv) $\overline{LY}(G)=-1080p^3+2280p^2-1160p-40$.

\end{thm}
\begin{proof}
By using edge partition from Table $2$, we calculate the new k-distance degree based topological indices for $R_{p}$ defined in equations $9-13$ as:
\item \textbf{(i) Leap Somber Index}
\begin{eqnarray*}
LSo(G)&=&\sum\limits_{uv\in E(G)}\sqrt{deg_2(u)+deg_2(v)}\\
&=& 4(5)^{\frac{1}{2}}+ 2(6)^{\frac{1}{2}}+ 8(7)^{\frac{1}{2}}+ 8(p-2)(8)^{\frac{1}{2}}+4(p-1)(10)^{\frac{1}{2}}+((p-1)^2+\\
& &2(p-1)(p-2))(12)^{\frac{1}{2}}\\
&=&4\sqrt{5}+2\sqrt{6}+8\sqrt{7}+16(p-2)\sqrt{2}+4(p-1)\sqrt{10}+(3p^2-8p+5)2\sqrt{3}\\
&=&4\sqrt{5}+2\sqrt{6}+8\sqrt{7}+16p\sqrt{2}-32\sqrt{2}+4p\sqrt{10}-4\sqrt{10}+6p^2\sqrt{3}-\\
& &16p\sqrt{3}+10\sqrt{3}\\
&=&6p^2\sqrt{3}+p(16\sqrt{2}+4\sqrt{10}-16\sqrt{3})+4(\sqrt{5}-\sqrt{10})+2(\sqrt{6}+4\sqrt{7}+\\
& &5\sqrt{3}-16\sqrt{2})
\end{eqnarray*}
\item \textbf{(iii) Hyper leap forgotten Index}
\begin{eqnarray*}
HLF(G)&=&\sum\limits_{uv\in E(G)}[deg^2_2(u)+deg^2_2(v)]^2\\
	&=&4(13)^2+2(18)^2+8(25)^2+8(p-2)(32)^2+4(p-1)(52)^2+(3p^2-8p+5)\\
& &(72)^2\\
	&=&676+648+5000+8192(p-2)+10816(p-1)+15552p^2-41472p\\
& &+25920\\
	&=&15552p^2-22464p+5044
\end{eqnarray*}
\item \textbf{(iv) Leap Y Index}
\begin{eqnarray*}
LY(G)&=&\sum\limits_{uv\in E(G)}[deg^3_2(u)+deg^3_2(v)]\\
	&=&4(8+27)+2(27+27)+8(27+64)+8(p-2)(64+64)+4(p-1)\\
& &(64+216)+((p-1)^2+2(p-1)(p-2))(216+216)\\
	&=&140+108+13824+32768(p-2)+85296(p-1)+1296p^2-3456p+\\
& &2160\\
	&=&1296p^2-1312p-32
\end{eqnarray*}
\item \textbf{(v) Leap Y coindex}
\begin{eqnarray*}
	 \overline{LY}(G)&=&(p-1)(k-F(G))-(k-Y(G))\\
	&=&(p-1)(216p^2-112p-72)-(1296p^2-1312p-32)\\
	&=&(p-1)(-1080p^2+1200p+40)\\
	&=&-1080p^3+2280p^2-1160p-40.
\end{eqnarray*}
where, $LF(G) = 216p^2-112p-72.$
\end{proof}
\section{Numerical Results and Discussions}
\noindent In this section, we show numerical results for the distance-degree based topological indices for zigzag and rhombic benzenoid systems. We compute numerical tables for distance-degree based indices such as first and second leap Zagreb index, the first and second leap hyper-Zagreb index, leap Somber index, hyper leap forgotten index, leap $Y$ index, and leap $Y$ coindex, for various values of $p$ ( Table 3-6). Moreover, we plot line graphs ( Figure 4-7) for some values of $p$ to investigate the behaviour of these topological indices.
\begin{table}[h!]
\centering
\caption{Numerical Computation of $LM_{1}(G)$, $LM_{2}(G)$, $HLM_{1}(G)$, $HLM_{2}(G)$ indices for $Z_{p}$} \label{table17}
\setlength{\tabcolsep}{15pt}
\renewcommand{\arraystretch}{1.0}
\begin{tabular}{l ccccc}
\hline
[p] & $LM_{1}(G)$ & $LM_{2}(G)$ & $HLM_{1}(G)$ & $HLM_{2}(G)$\\
\hline
2 & 148 & 273 & 1068 & 4417\\
3 & 231 & 446 & 1777 & 7754\\
4 & 314 & 619 & 2486 & 11091\\
5 & 397 & 792 & 3195 & 14428\\
6 & 480 & 965 & 3904 & 17765\\
7 & 563 & 1138 & 4613 & 21102\\
8 & 646 & 1311 & 5322 & 24439\\
9 & 729 & 1484 & 6031 & 27776\\
10 & 812 & 1657 & 6740 & 31113\\
\hline
\end{tabular}
\end{table}

\begin{figure}[H]
  \centering
  \includegraphics[width=4in]{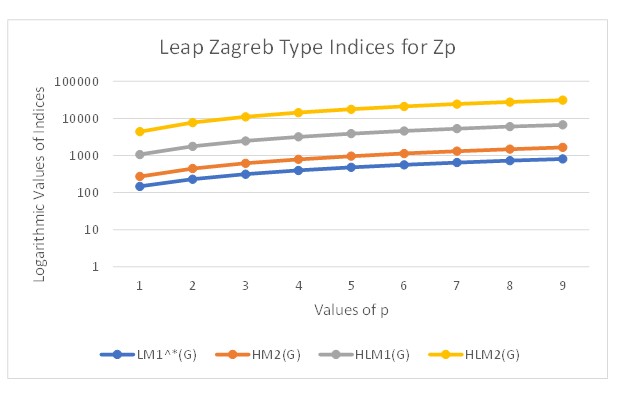}
  \caption{Leap Zagreb type indices for $z_{p}$}\label{Figure4}
\end{figure}

\begin{table}[h!]
\centering
\caption{Numerical Computation of $LSO(G)$, $LF(G)$, $HLF(G)$, $LY(G)$, $\overline{LY}(G)$ indices for $Z_{p}$} \label{table17}
\setlength{\tabcolsep}{15pt}
\renewcommand{\arraystretch}{1.0}
\begin{tabular}{l cccccc}
\hline
[p] & $LSO(G)$ & $LF(G)$ & $HLF(G)$ & $LY(G)$ & $\overline{LY}(G)$\\
\hline
2 & 55.21 & 572 & 19438 & 2380 & -1808\\
3 & 83.9 & 935 & 33891 & 4035 & -6200\\
4 & 112.59 & 1298 & 48344 & 5640 & -13176\\
5 & 141.28 & 1661 & 62797 & 7345 & -22736\\
6 & 169.97 & 2024 & 77250 & 9000 & -34880\\
7 & 198.66 & 2387 & 91703 & 10655 & -49608\\
8 & 227.35 & 2750 & 106156 & 12310 & -66920\\
9 & 256.04 & 3113 & 120609 & 13965 & -86816\\
10 & 284.73 & 3476 & 135062 & 15620 & -109296\\
\hline
\end{tabular}
\end{table}
\begin{figure}[H]
  \centering
  \includegraphics[width=4in]{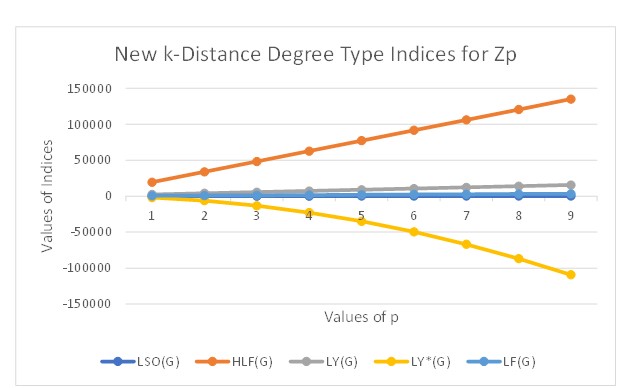}
  \caption{New k-distance degree type indices for $z_{p}$}\label{Figure5}
\end{figure}
\begin{table}[h!]
\centering
\caption{Numerical Computation of $LM_{1}(G)$, $LM_{2}(G)$, $HLM_{1}(G)$, $HLM_{2}(G)$ indices for $R_{p}$} \label{table17}
\setlength{\tabcolsep}{15pt}
\renewcommand{\arraystretch}{1.0}
\begin{tabular}{l ccccc}
\hline
[p] & $LM_{1}(G)$ & $LM_{2}(G)$ & $HLM_{1}(G)$ & $HLM_{2}(G)$\\
\hline
2 & 140 & 270 & 1108 & 5058\\
3 & 328 & 746 & 3028 & 18482\\
4 & 588 & 1438 & 5812 & 39682\\
5 & 920 & 2346 & 9460 & 68658\\
6 & 1324 & 3470 & 13972 & 105410\\
7 & 1800 & 4810 & 19348 & 149938\\
8 & 2348 & 6366 & 25588 & 202242\\
9 & 2968 & 8138 & 32692 & 262322\\
10 & 3660 & 10126 & 40660 & 330178\\
\hline
\end{tabular}
\end{table}

\begin{figure}[H]
  \centering
  \includegraphics[width=4in]{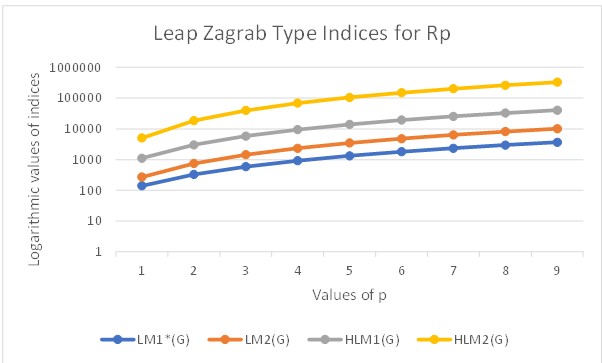}
  \caption{Leap Zagreb type indices for $R_{p}$}\label{Figure6}
\end{figure}

\begin{table}[h!]
\centering
\caption{Numerical Computation of $LSO(G)$, $LF(G)$, $HLF(G)$, $LY(G)$, $\overline{LY}(G)$ indices for $R_{p}$} \label{table17}
\setlength{\tabcolsep}{15pt}
\renewcommand{\arraystretch}{1.0}
\begin{tabular}{l cccccc}
\hline
[p] & $LSO(G)$ & $LF(G)$ & $HLF(G)$ & $LY(G)$ & $\overline{LY}(G)$\\
\hline
2 & 51.11 & 568 & 22324 & 2528 & 15400\\
3 & 110.62 & 1536 & 77620 & 7696 & 46160\\
4 & 190.91 & 2936 & 164020 & 15456 & 100920\\
5 & 291.23 & 4768 & 281524 & 25808 & 186160\\
6 & 413.83 & 7032 & 430132 & 38752 & 308360\\
7 & 556.46 & 9728 & 609844 & 54288 & 474000\\
8 & 719.87 & 12856 & 820660 & 72416 & 689560\\
9 & 904.06 & 16416 & 1062580 & 93136 & 961520\\
10 & 1109.04 & 20408 & 1335604 & 116448 & 1296360\\
\hline
\end{tabular}
\end{table}
\begin{figure}[H]
  \centering
  \includegraphics[width=3.4in]{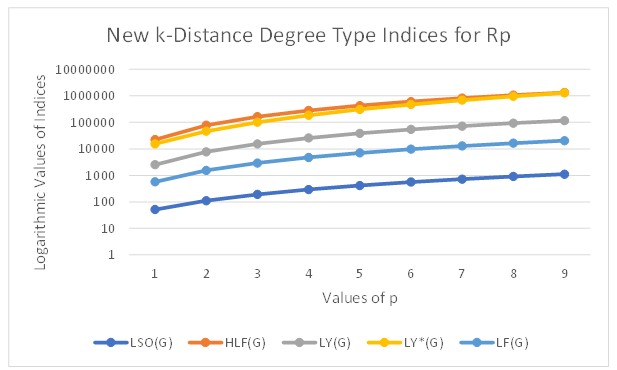}
  \caption{New k-distance degree type indices for $R_{p}$}\label{Figure7}
\end{figure}
\noindent From Figures 4, 5, 6, and 7, we determined that as $p$ increases there is an increase in all the indices except leap $Y$ coindex of $Z_p$. The leap $Y$ coindex of $Z_p$ decreases as the values of $p$ increasing.\\

\noindent The modified first Zagreb connection index is a chemical descriptor that first appears in a $1972$ for the calculation of the total electron energy of alternant hydrocarbons \cite{S}. According to research, the first Zagreb connection index performs better as a predictor of entropy, enthalpy of vaporisation, standard enthalpy of vaporisation, and acentric factor of octane isomers than the other Zagreb connection indices.\\

\noindent From Table $3$ and Figure $4$, we analyze that for a zigzag benzenoid system the second hyper-leap Zagreb index attains maximum predictive ability than the other three leap Zagreb indices.\\

\noindent A dataset of octane isomers was used to assess the forgotten index's predictive ability. Moreover, it was discovered that the forgotten-predictive index's power is pretty comparable to first zagreb index \cite{17}. The zigzag benzenoid system and the rhombic benzenoid system have higher leap and hyper leap forgotten index then the first leap and first hyper leap zagreb index.\\

\noindent In mathematical chemistry, the somber indices are very predictive. These can be acquired for the graph energy upper and lower bounds, as well as the relationship between the Sombor index and graph energy for molecular graph types \cite{A}. The leap somber index for these two molecular graphs increases as $p$ increases.\\

\noindent The Y-coindex is one of the most useful correlation indices for comprehending the physicochemical properties of octane isomers \cite{21}. From Table $4$ and Figure $5$, we analyse that the leap Y- index attains maximum value as $p$ increases and leap Y- coindex attains minimum value as $p$ increases for zigzag benzenoid system.\\ 
\noindent  Moreover, for rhombic benzenoid system, the leap Y-index and leap Y-coindex attains maximum value as $p$ increases (see Table $6$ \& Figure $7$).\\

\section{Conclusion and Future Work}
\noindent The topological indices can be used to identify the physicochemical characteristics of chemical compounds. In this work, we utilize edge partition technique for obtaining results related to distance-degree based topological indices such as first and second leap Zagreb indices, first and second leap hyper-Zagreb indices for the zigzag and rhombic benzenoid system. Furthermore, the expression for new $k$-distance degree-based topological indices such as leap Somber index, hyper leap forgotten index, leap $Y$ index and leap $Y$ coindex of the zigzag and rhombic benzenoid system are derived. Also, we obtained their numerical results and plotted the graphs of these indices for some values of $p$ in order to determine the significance of the physicochemical characteristics of the zigzag and rhombic benzenoid systems.\\

 We are mention some possible directions for future research as multiplicative k-distance degree-based topological indices for certain molecular graphs and to determine the predictive ability of physio-chemical characteristics in the case of a dataset of octane isomers. Also, one can find the bounds for the newly defined $k$-distance degree-based topological indices.\\
 
\noindent \textbf{Novelty statement}\\\\Graphs invariants plays an important role to analyze the abstract structures of molecular graph of chemical compounds. Topological index is one of the graph invariant that describes the topology of a chemical compounds based on a molecular structure. Various distance-degree based topological indices of chemical graphs has been recently computed. But there are still many chemical compounds for which distance-degree based topological indices has not been found yet. Therefore, in this article, we compute some new distance-degree based topological indices for the molecular graph of certain classes of benzenoid system.
Furthermore, we compute their numerical results and plot their line graphs for the comparision of these indices. \\\\
\noindent \textbf{Data Availibility}\\
No data were used to support the finding of this study.\\

\noindent \textbf{Conflicts of Interest}\\
There are no conflicts of  interest declared by the authors.\\

\noindent \textbf{Conflicts of Interest:}\\
There are no conflicts of interest to be declared by the authors.\\
 
\noindent \textbf{Funding:}\\
Not Applicable. No funds have been received.\\

\noindent \textbf{Authors contributions:}\\ First draft was prepared by Sohan Lal and Vijay Kumar Bhat; figures have been prepared by Sohan Lal and Karnika Sharma; all authors have reviewed the final draft.\\

\end{document}